\documentclass[psamsfonts]{amsart}

\usepackage{amssymb,amsfonts}
\usepackage[all,arc]{xy}
\usepackage{enumerate}
\usepackage{mathrsfs}
\usepackage[mathscr]{euscript}
\usepackage{graphicx}
\usepackage{float}
\usepackage{url}
\usepackage{color}

\newtheorem{thm}{Theorem}[section]
\newtheorem{thmLetter}{Theorem}
\newtheorem{cor}[thm]{Corollary}
\newtheorem{prop}[thm]{Proposition}
\newtheorem{lem}[thm]{Lemma}

\theoremstyle{definition}
\newtheorem{defn}[thm]{Definition}

\theoremstyle{remark}
\newtheorem{rem}[thm]{Remark}

\makeatletter
\let\c@equation\c@thm
\makeatother
\numberwithin{equation}{section}

\newcommand{\IC}{\mathbb{C}}

\newcommand{\IP}{\mathbb{P}}
\newcommand{\IQ}{\mathbb{Q}}
\newcommand{\IR}{\mathbb{R}}

\newcommand{\IZ}{\mathbb{Z}}

\newcommand{\sA}{\mathcal{A}}
\newcommand{\sB}{\mathcal{B}}
\newcommand{\sC}{\mathcal{C}}

\newcommand{\sM}{\mathcal{M}}

\newcommand{\sP}{\mathcal{P}}

\newcommand{\sT}{\mathcal{T}}
\newcommand{\sU}{\mathcal{U}}
\newcommand{\sV}{\mathcal{V}}

\newcommand{\gap}{\hspace{15pt}}

\DeclareMathOperator{\Aut}{Aut}

\DeclareMathOperator{\Fix}{Fix}
\DeclareMathOperator{\Mod}{\Gamma}

\DeclareMathOperator{\Teich}{\sT}

\newcommand{\abc}{\mathcal{ABC}}

\begin{document}

\title{On the topology of moduli spaces of real algebraic curves}

\author{Alex Pieloch}
\address{Department of Mathematics\\ Columbia University\\ 2990 Broadway \\
New York, NY 10027}
\email{pieloch@math.columbia.edu}

\thanks{The author is supported by the National Science Foundation Graduate Student Fellowship Program through grant DGE 16-44869. He was also partially supported by the National Science Foundation through grant DMS-1406420 during the Summer of 2017 during which this paper was written.}

%\date{}

\maketitle

\begin{abstract}
We show that mapping class groups associated to all types of real algebraic curves are virtual duality groups.  We also deduce some results about the orbifold homotopy groups of the moduli spaces of real algebraic curves.  We achieve these results by defining a new complex associated to a not necessarily orientable surface with boundary called the $\abc$-complex.  This complex encodes the intersection patterns of isotopy classes of essential simple arcs, boundary components, and essential simple closed curves.  We show that this complex is homotopy equivalent to an infinite wedge of spheres; its dimension is dependent on the topological type of the surface.  
\end{abstract}

\tableofcontents

\section{Introduction}\label{Intro}
We are interested in studying the topology of moduli spaces of real algebraic curves.  In this paper, we define a real algebraic curve to be a pair $(X,\sigma)$ where $X$ is a smooth, compact, connected Riemann surface and $\sigma$ is an anti-holomorphic involution of $X$.  A $2n$-pointed real algebraic curve is a real algebraic curve $(X,\sigma)$ along with finite collection of $2n$ marked points such that $\sigma$ acts freely on the collection of marked points.  Up to conjugation by homeomorphisms there is a finite number of topological types of such orientation reversing involutions.

Forgetting the complex structure of a Riemann surface $X$ produces a real oriented surface $S$.  Similarly, forgetting the real structure of a real algebraic curve produces a real oriented surface along with an orientation reversing involution.  When one studies the moduli spaces of Riemann surfaces, it is natural to define the mapping class group of a real oriented surface $S$ of genus $g$ with $n$ marked points, which we denote by $\Mod_{g,n}$.  In our case, it is natural consider the mapping class group of a real oriented surface $S$ of genus $g$ with $2n$ marked points and an orientation reversing involution $\sigma$, denoted $\Mod_{g,2n}^\sigma$.  It is defined to be the group of isotopy classes of orientation preserving diffeomorphisms of $S$ that commute with $\sigma$ and leave the collection of marked points invariant.  Harer \cite{Harer_VirtualCohomDim} showed that $\Mod_{g,n}$ is a virtual duality group in the sense of Bieri and Eckmann \cite{BieriEckmann_GroupHomDuality}.  Our first main result shows that the mapping class group associated to each topological type of real algebraic curve, that is, each $\Mod_{g,2n}^\sigma$, is also a virtual duality group.

To prove this result, we employ a strategy similar to that of Harer's.  For the case of $\Mod_{g,n}$, Harer studied the topology of the curve complex of a surface $S$ of genus $g$ with $n$ marked points, denoted $\sC(S)$.  To prove the result for $\Mod_{g,2n}^\sigma$, we consider the simplicial complex $\sC^\sigma(S)$ whose set of $k$-simplices consists of ($k+1$)-tuples of $\sigma$-orbits of isotopy classes of $\sigma$-invariant\footnote{A curve $\gamma$ in $S$ is \emph{$\sigma$-invariant} if $\sigma(\gamma)$ and $\gamma$ may be realized disjointly.  Notice that simple closed curves fixed by $\sigma$ are $\sigma$-invariant since two simple closed curves in the same isotopy class may be realized disjointly.}  essential simple closed curves in $S$ that may be realized disjointly.  The complex $\sC^\sigma(S)$ may be regarded as a tits building associated to real algebraic curves of type $\sigma$.  The main technical result of this paper is computing the homotopy type of this complex, which we show to be an infinite wedge of spheres of a particular dimension dependent on the topological type of the pair $(S,\sigma)$.  From this, we obtain our first main result.

\begin{thmLetter}\label{thm:InvolutionDualityGroup}
Let $S$ be a surface of genus $g \geq 3$ with $2n$ marked points and an orientation reversing involution $\sigma$ that has $0 \leq k \leq g+1$ fixed circles.  The group $\Mod_{g,2n}^\sigma$ is a virtual duality group of dimension $vcd(g,n,k)$ where 
\[ vcd(g,n,k) = \begin{cases} 
2g-3 & \mbox{if     }\,\,\,n=0, \,\,\,k=0, \\
2g+n-2 & \mbox{if     }\,\,\, n>0, \,\,\, k=0,\\
g+n-2 & \mbox{if     } \,\,\,k=g+1,\\
2g+n-k-2 & \mbox{if     } \,\,\,0<k<g+1.\\
\end{cases}\]
\end{thmLetter}

The group $\Mod_{g,2n}^\sigma$ acts properly discontinuously and virtually freely on a contractible subspace of the Teichm\"{u}ller space of $S$.  This subspace may be thought of as the Teichm\"{u}ller space of a surface $S$ with an orientation reversing involution $\sigma$.  The quotient of this subspace by $\Mod_{g,2n}^\sigma$ is the moduli space of real algebraic curves of genus $g$ with $2n$ marked points and topological type of involution $\sigma$.  We denote it by $\sM_{g,2n}^\sigma$.  Since $\Mod_{g,2n}^\sigma$ acts virtually freely and properly discontinuously on $\sM_{g,2n}^\sigma$, we have that $\sM_{g,2n}^\sigma$ is an orbifold classifying space for $\Mod_{g,2n}^\sigma$ and for each $i$ we have an isomorphism
\[H^i(\Mod_{g,2n}^\sigma;\IQ) \cong H^i(\sM_{g,2n}^\sigma;\IQ).\]
Several fundamental properties of these moduli spaces have been investigated by Buser, Sepp\"{a}l\"{a}, and Silhol \cite{BuserSeppala_SymmetricPantsDecompostionsOfRiemannSurfaces,SeppalaSilhol_ModuliSpacesForRealAlgebraicCurvesAndRealAbelianVarieties,Seppala_ComplexAlgebraicCurvesWithRealModuli}.  Furthermore, the virtual Euler characteristics of these moduli spaces have been computed by Goulden, Harer, and Jackson \cite{GouldenHarerJackson_VirtualEulerChar}.

Using an ascending sequence of compact subspaces of $\sM_{g,2n}^\sigma$ that exhaust $\sM_{g,2n}^\sigma$, we are able to define a neighborhood of infinity of $\sM_{g,2n}^\sigma$, which we denote by $\sU_{g,2n}^\sigma$.  Using the connectivity of $\sC^\sigma(S)$, we show that the $m$th orbifold homotopy group of $\sM_{g,2n}^\sigma$ relative to infinity is trivial for all $m$ less than the homological dimension of $\sC^\sigma(S)$.  More precisely, we have our second main result.

\begin{thmLetter}\label{thm:OrbiHomotopyGroups}
Let $S$ be a surface of genus $g$ with $2n$ marked points and an orientation reversing involution $\sigma$ that has $0 \leq k \leq g+1$ fixed circles.  The inclusion $\sU_{g,2n}^\sigma \hookrightarrow \sM_{g,2n}^\sigma$ induces an isomorphism $\pi_m^{orb}(\sU_{g,2n}^\sigma) \cong \pi_m^{orb}(\sM_{g,2n}^\sigma)$ for each $m<d(g,n,k)$, where
\[ d(g,n,k) = \begin{cases} 
g-1 & \mbox{if     }\,\,\,n=0, \,\,\,k=0, \\
g+n-2 & \mbox{if     }\,\,\, n>0, \,\,\, k=0,\\
2g+n-2 & \mbox{if     }\,\,\,k=g+1,\\
g+n+k-2 & \mbox{if     }\,\,\,0<k<g+1.\\
\end{cases}\]

\end{thmLetter}

The remainder of this paper is organized as follows.  In section \ref{sect:BackRAC}, we discuss some preliminary material on real algebraic curves and mapping classes groups associated to them.  In section \ref{sect:TeichTheory}, we use Teichm\"{u}ller theory to construct moduli spaces of real algebraic curves.  After defining these moduli spaces, we study ascending sequences of compact subspaces of these moduli spaces and define a subspace at infinity.  In section \ref{sect:CurveComplexes}, we first study curve complexes of (possibly non-orientable) surfaces.  Next, we introduce the $\abc$-complex of a surface and compute its homotopy type.  In section \ref{sect:VCD}, we briefly review the work of Bieri and Eckmann on duality groups and give the proof of Theorem \ref{thm:InvolutionDualityGroup}.  In section \ref{sect:OrbiHomotopyGroups}, we review the definition of orbifold homotopy groups and give the proof of Theorem \ref{thm:OrbiHomotopyGroups}.

\subsection*{Acknowledgements}
I would like to thank my undergraduate research advisor, Richard Hain, for all of his time, expertise, and guidance through out my time at Duke as well as his comments, critiques, and suggestions on numerous drafts of this paper.  I would also like to thank John Harer for several enlightening discussions about curve complexes.

\section{Background on Real Algebraic Curves}\label{sect:BackRAC}

In this section, we discuss some preliminary material on real algebraic curves and mapping class groups associated to them.

\subsection{Real Algebraic Curves}
\begin{defn}\label{defn:RAC}
A \emph{$2n$-pointed (smooth, projective) real algebraic curve of genus $g$} is a compact, connected Riemann surface $X$ of genus $g$ with $2n$ marked points along with an anti-holomorphic involution $\sigma$ of $X$ that acts freely on the collection of marked points.  We denote such a real algebraic curve by the pair $(X,\sigma)$.
\end{defn}

\begin{defn}\label{defn:RACIsom}
Two real algebraic curves $(X,\sigma)$ and $(X',\sigma')$ are isomorphic if there exists a biholomorphism $f: X \to X'$ such that $f \circ \sigma = \sigma' \circ f$ and $f$ is a bijection between the collections of marked points of $X$ and $X'$.
\end{defn}

The sense in which a compact, connected Riemann surface $X$ with an anti-homomorphic involution $\sigma$ is a real algebraic curve is described by the following proposition.

\begin{prop}\label{prop:RACNormIrred}
A compact Riemann surface $X$ with an anti-holomorphic involution $\sigma$ is isomorphic to the compact Riemann surface whose field of meromorphic functions is the fraction field of $\IC[x,y]/(f)$ for an irreducible polynomial $f$ in $\IR[x,y]$.
\end{prop}
\begin{proof}
Let $F$ be the field of meromorphic functions on $X$. Then $\sigma$ induces the field involution $s : F \to F$ given by $s(\phi) = \overline{\sigma^* \phi}$.  The restriction of $s$ to the constant functions is complex conjugation.  Let $K$ denote the fixed field of $s$. We have that $[F:K] = 2$.  Since $K$ contains the real numbers $\IR$ but not the complex numbers $\IC$, we have that $K \cap \IC$ is $\IR$.  It follows that $K \otimes_\IR \IC$ is $F$.

Since $F$ has transcendence degree $1$ over $\IC$, we have that $K$ has transcendence degree $1$ over $\IR$, that is, $K$ is an algebraic extension of $\IR(x)$. Let $y$ be any primitive generator of $K$ over $\IR(x)$ and $f(x,y)$ its minimal polynomial, which is an irreducible element of $\IR(x)[y]$.  By clearing denominators, we may assume that $f(x,y)$ is in $\IR[x,y]$ and is irreducible. Since $F = K \otimes_\IR \IC$, we have that $F = \IC(x)[y]/(f)$, which implies that $X$ is the normalization of the plane curve in $\IP^2(\IC)$ defined by $f(x,y) = 0$.
\end{proof}

\begin{defn}\label{defn:RACTopType}
Two real algebraic curves $(X,\sigma)$ and $(X',\sigma')$ are \emph{topologically equivalent} (or of the same \emph{topological type}) if there exists a homeomorphism $f: X \to X'$ such that $f \circ \sigma = \sigma' \circ f$ and $f$ is a bijection between the collections of marked points of $X$ and $X'$.
\end{defn}

We consider a real algebraic curve $(X,\sigma)$ of genus $g$ (ignoring marked points for the moment) and set
\[\Fix(\sigma) = \{ x \in X \mid \sigma(x) = x\}.\]
It was originally shown by Harnack \cite{Harnack_HarnackInequality} that $\Fix(\sigma)$ is a smooth, 1-dimensional, real submanifold of $X$ and is the disjoint union of a collection of $k$ circles where $0 \leq k \leq g+1$.  We may consider the surface obtained by taking the complement of the fixed point set of $\sigma$ in $X$, that is, $X - \Fix(\sigma)$.  Define the invariant $a \in \{0,1\}$ associated to $(X,\sigma)$ by
\[ a = \begin{cases} 0 & \mbox{if }X-\Fix(\sigma) \mbox{ is not connected,}\\1 & \mbox{if }X-\Fix(\sigma) \mbox{ is connected.}\\ \end{cases}\]
Let $k$ denote the number of connected components of $\Fix(\sigma)$.  The topological type of a real algebraic curve is completely determined by the tuple $(g,k,a)$.

\begin{thm}[Weichold \cite{Weichold_InvolutionTopType}]\label{thm:InvolutionTopType}
Let $(X,\sigma)$ and $(X',\sigma')$ be two real algebraic curves.  The following conditions are equivalent:
\begin{enumerate}
\item There exists a homeomorphism $f: X \to X'$ such that $f \circ \sigma = \sigma' \circ f$.
\item $X/\sigma$ is homeomorphic to $X'/\sigma'$.
\item $(g,k,a) = (g',k',a')$.
\end{enumerate}
The tuple $(g,k,a)$ is subject to the following conditions:
\begin{enumerate}
\item if $a = 0$, then $1 \leq k \leq g+1$ and $k \equiv (g+1) \,\,(\mbox{\emph{mod }} 2)$,
\item if $a = 1$, then $0 \leq k \leq g$.
\end{enumerate}
Moreover, for each tuple $(g,k,a)$ that satisfies the conditions above, there exists a real algebraic curve of topological type $(g,k,a)$.
\end{thm}

Theorem \ref{thm:InvolutionTopType} implies that there are $\lfloor \frac{3g+4}{2} \rfloor$ topological types of $2n$-pointed real algebraic curves of genus $g$ for each $n$.  Given a real algebraic curve $(X,\sigma)$ of topological type $(g,g+1,0)$, the quotient $X/\sigma$ is homeomorphic to a sphere with $g+1$ boundary components.  For the topological type $(g,0,1)$, we have that $X/ \sigma$ is homeomorphic to a non-orientable surface with $g+1$ cross caps.\footnote{A non-orientable surface with $c$ cross caps is, by definition, the connected sum of $c$ copies of $\IR\IP^2$.}  In general, $X \to X/\sigma$ is a ramified double covering of the not necessarily orientable surface $X/\sigma$, which is ramified over the boundary components of $X/\sigma$.  Hence, to each topological type of real algebraic curve, we may associate the ramified covering $X \to X/\sigma$ of real surfaces.  In light of this, we may think of a real algebraic curve as a ramified covering of real surfaces endowed with a conformal structure.

\subsection{Mapping Class Groups}
By forgetting the real structure of a real algebraic curve, we obtain a surface with an orientation reversing involution.  Recall that the mapping class group of an orientable surface S of genus $g$ with $n$ marked points $p_1,\dots,p_n$, denoted $\Mod_{g,n}$, is the group of isotopy classes of orientation preserving diffeomorphisms of $S$ and leave $\{p_1,\dots,p_n\}$ invariant.

\begin{defn}\label{defn:RACMCG}
The \emph{mapping class group} of a surface $S$ of genus $g$ with $2n$ marked points $p_1,\dots,p_{2n}$ and an orientation reversing involution $\sigma$, denoted $\Mod_{g,2n}^\sigma$, is the group of isotopy classes of orientation preserving diffeomorphisms of $S$ that commute with $\sigma$ up to isotopy and leave $\{p_1,\dots,p_{2n}\}$ invariant.
\end{defn}

One may relate $\Mod_{g,2n}^\sigma$ to a mapping class group of the quotient $S/\sigma$.  To do this, we first recall a fact from Riemannian geometry.  Given a collection of isometries $\mathscr{I}$ of a Riemannian manifold $M$, the fixed locus of $\mathscr{I}$,
\[ \Fix(\mathscr{I}) = \{ p \in M \mid g(p) = p \,\,\,\, \forall \, g \in \mathscr{I}\}, \]
is a totally geodesic submanifold (possibly not connected, but with finitely many path components) of $M$.  Furthermore, if $M$ is compact, then each path component of $\Fix(\mathscr{I})$ is compact (see \cite[II.5]{Kobayashi_TransformationGroups}).  

We consider the case where $M$ is $S$ and $\mathscr{I}$ is $\sigma$.  Let $ds^2$ be a Riemannian metric on $S$.  Averaging over $\sigma$, we have that 
\[ \frac{1}{2} \left( ds^2+ \sigma^* ds^2 \right)\]
is a Riemannian metric on $S$ where $\sigma$ acts by isometry.  Since $\sigma$ is an orientation reversing involution, it locally acts by reflection.  We have that $\Fix(\sigma)$ is a disjoint collection of embedded circles and consequently $S/\sigma$ is a surface with boundary.  Let $\pi: S \to S/\sigma$ denote the quotient map.  Each diffeomorphism of $S/\sigma$ has a unique orientation preserving lift to $S$ which commutes with $\sigma$.  Hence, we may identify $\Mod_{g,2n}^\sigma$ with the group of isotopy classes of diffeomorphism of $S/\sigma$ that leave the marked points invariant, preserve a local orientation of $S/\sigma$ at each $\pi(p_i)$, can permute the boundary components of $S/\sigma$, and are not required to fix the boundary components pointwise.  Moreover, isotopies of diffeomorphisms are also not required to fix the boundary components pointwise.

\section{Teichm\"{u}ller Theory of Real Algebraic Curves}\label{sect:TeichTheory}

In this section, we first construct moduli spaces of real algebraic curves.  Next, using an ascending sequence of compact subspaces, we define a notion of infinity for our moduli spaces.  In this section, we assume that $S$ is a compact, oriented surface of genus $g$ with $n$ marked points and with empty boundary.  We let $P = \{p_1,\dots,p_n\}$ denote the finite collection of marked points in $S$ and set $S' = S-P$.  We denote this pair by $(S,P)$ and refer to $(S,P)$ as a \emph{$n$-pointed surface}.  We also assume that $2-2g-n = \chi(S')<0$.  

\subsection{Construction of Moduli Spaces}\label{subsect:TeichTheory_Construction}
While constructing the moduli spaces of real algebraic curves, we will make use of the equivalence between complex structures and oriented hyperbolic structures that follows from the Uniformization Theorem.

\begin{defn}\label{defn:HyperbolicStructure}
A \emph{hyperbolic structure} on a $n$-pointed surface $(S,P)$ is a $n$-pointed surface $(X,Q)$, where $X-Q$ is a surface with a complete hyperbolic metric, and diffeomorphism $f: (S,P) \to (X,Q)$ such that $f$ maps $P$ bijectively to $Q$.  The diffeomorphism $f$ is called a \emph{marking} of $(X,Q)$.
\end{defn}

\begin{defn}\label{defn:IsotopicHyperStr}
Two hyperbolic structures on a pointed surface $(S,P)$, $f:(S,P) \to (X,Q)$ and $g: (S,P) \to (Y,R)$, are \emph{isotopic} if there exists an isometry $I: X' \to Y'$ such that the following diagram commutes up to isotopy.
\[\xymatrix{S' \ar[r]^{f} \ar[dr]_{g} & X' \ar[d]^{I}\\&Y'}\]
Denote this equivalence relation by $\sim$.
\end{defn}

\begin{defn}\label{defn:TeichSpace}
The \emph{Teichm\"{u}ller space} of a $n$-pointed surface $(S,P)$ of genus $g$, denoted $\Teich_{g,n}$, is the set of isotopy classes of hyperbolic structures on $(S,P)$:
\[ \Teich_{g,n} := \{ \mbox{hyperbolic structures on }(S,P)\} / \sim.\]
\end{defn}

Denote an equivalence class of hyperbolic structures on $(S,P)$ by $\mathscr{X} = [f : (S,P) \to (X,Q)]$.  Recall that $\Teich_{g,n}$ may be endowed with the structure of a smooth manifold and is diffeomorphic to $\IR^{6g-6+2n}$.  The mapping class group of a surface $S$ with a collection of marked points $P$ acts on $\Teich_{g,n}$.  Given $\phi \in \Mod_{g,n}$ and $[f: (S,P) \to (X,Q)] \in \Teich_{g,n}$, we obtain a new hyperbolic structure on $(S,P)$ by considering $[f \circ \phi^{-1}: (S,P) \to (X,Q)]$.  This $\Mod_{g,n}$-action is properly discontinuous and virtually free.\footnote{It is well-known that the kernel of the homomorphism $\Mod_{g,n} \to \Aut(H_1(S;\IZ/3\IZ))$ obtained from the obvious action of $\Mod_{g,n}$ on $H_1(S;\IZ/3\IZ)$ is a finite index, torsion free subgroup that acts freely on $\Teich_{g,n}$.}  We define the \emph{moduli space of $n$-pointed Riemann surfaces of genus $g$} as 
\[\sM_{g,n} :=\Mod_{g,n} \backslash \Teich_{g,n}.\]  

To construct the moduli spaces of real algebraic curves, we proceed in a similar manner.  We will now let $\{p_1,\dots,p_{2n}\}$ be the collection of marked points $P$ in $S$. 

Notice that each hyperbolic structure $[f:(S,P) \to ((X,Q),\tau)]$ where $((X,Q),\tau)$ is a pointed real algebraic curve with anti-holomorphic involution $\tau$ determines an isotopy class of orientation reversing involutions $\sigma: S \to S$ that leaves the marked points invariant, namely, $\sigma := f^{-1} \circ \tau \circ f$.  For each orientation reversing involution $\sigma$ of $S$ there is a corresponding involution of $\Teich_{g,2n}$, denoted $\Psi_\sigma : \Teich_{g,2n} \to \Teich_{g,2n}$, which is given by
\[[f: (S,P) \to (X,Q)] \mapsto [f \circ \sigma: (S,P) \to (\overline{X},\overline{Q})]\]
where $\overline{\cdot}$ denotes complex conjugation on $X$.  A hyperbolic structure $[f:(S,P) \to (X,Q)]$ is fixed by $\Psi_\sigma$ if and only if $X$ has a real algebraic structure $\tau: X \to X$ such that $\tau \simeq f \circ \sigma \circ f^{-1}$.  Denote the fixed locus of $\Psi_\sigma$ in $\Teich_{g,2n}$ by $\Teich_{g,2n}^\sigma$.  Fixing an appropriate $\sigma$-invariant pants decomposition, an argument using Fenchel--Nielsen coordinates shows that $\Teich_{g,2n}^\sigma$ is homeomorphic to $3g-3+2n$.  Given $\phi \in \Mod_{g,2n}$, $\Teich_{g,2n}^\sigma$ is invariant under the action of $\phi$ if and only if $\phi$ commutes with $\sigma$, that is, $\phi \in \Mod_{g,2n}^\sigma$.  Since $\Mod_{g,2n}$ acts properly discontinuously and virtually freely on $\Teich_{g,2n}$, we have that $\Mod_{g,2n}^\sigma$ acts properly discontinuously and virtually freely on $\Teich_{g,2n}^\sigma$.  We consider the orbit space
\[\sM_{g,2n}^\sigma := \Mod_{g,2n}^\sigma \backslash \Teich_{g,2n}^\sigma.\]
If two orientation reversing involutions of $S$ are conjugate, then there exists a natural identification of their moduli spaces.  With this in mind, we make the following definition.

\begin{defn}\label{defn:ModuliRAC}
The \emph{moduli space of $2n$-pointed real algebraic curves of genus $g$}, denoted $\sM_{g,2n}^{\IR}$, is the space of isomorphism classes of $2n$-pointed real algebraic curves of genus $g$ and is given by
\[ \sM_{g,2n}^\IR = \coprod_{\sigma \in \Sigma} \sM_{g,2n}^\sigma\]
where $\Sigma$ denotes the set of conjugacy classes of orientation reversing involutions of $S$ that leave the marked points of $S$ invariant.
\end{defn}

%%%%%%%%%%%%%%%%%%%%%%%%%%%%%%
\subsection{Compact Subspaces}\label{subsect:TeichTheory_Compact}
In later sections, it will be useful for us to consider a compact subspace of $\sM_{g,2n}^\sigma$.  In order to obtain such a subspace, we will construct a subspace of $\Teich_{g,2n}^\sigma$ whose image in $\sM_{g,2n}^\sigma$ is compact.  

Given a hyperbolic structure $\mathscr{X} = [f: (S,P) \to (X,Q)]$ in $\Teich_{g,n}$ and an isotopy class of simple closed curves $a$ in $S$, let $\ell_{\mathscr{X}}(a)$ denote the length of the unique geodesic in the isotopy class $f(a)$ with respect to the hyperbolic metric on $X$.  Recall the following fact \cite[II.3.3]{Abikoff_RealAnalyticTeichTheory}.

\begin{prop}\label{prop:ShortGeodesics}
Let $a$ and $b$ be isotopy classes of simple closed curves in a surface $S$ with a collection of marked points $P$ and let $\mathscr{X}$ be a hyperbolic structure on $(S,P)$.  There exists $\delta >0$ such that if $\ell_\mathscr{X}(a),\ell_\mathscr{X}(b)\leq \delta$, then the geodesic representatives in the isotopy classes of $a$ and $b$ are either disjoint or coincide.
\end{prop}

Recall that the \emph{curve complex} of a surface $S$, denoted $\sC(S)$, is the simplicial complex whose $k$-simplices are $(k+1)$-tuples of isotopy classes of essential simple closed curves in $S$ that may be realized disjointly.  Interpreting Proposition \ref{prop:ShortGeodesics} in terms of simplices in $\sC(S)$, we have the following corollary.

\begin{cor}\label{cor:ShortSimpsCC}
Let $a_0,\dots,a_k$ be isotopy classes of essential simple closed curves in a surface $S$.  If $\ell_\mathscr{X}(a_0),\dots,\ell_\mathscr{X}(a_k) \leq \delta$, then  $a_0,\dots,a_k$ span a $(k+1)$-simplex in $\sC(S)$.
\end{cor}

Fix $\varepsilon<\delta$, where $\delta$ is given in Proposition \ref{prop:ShortGeodesics}.  Define the \emph{truncated Teichm\"{u}ller space}, denoted $\Teich_{g,n}(\varepsilon)$, as
\[ \Teich_{g,n}(\varepsilon) := \{ \mathscr{X} \in \Teich_{g,n} \mid \ell_\mathscr{X}(a) \geq \varepsilon \hspace{4pt} \forall a \in \sC(S)\}.\]

It is clear that $\Teich_{g,n}(\varepsilon)$ is invariant under the action of $\Mod_{g,n}$.  Define the \emph{truncated moduli space}, $\sM_{g,n}(\varepsilon)$, by
\[\sM_{g,n}(\varepsilon) := \Mod_{g,n} \backslash \Teich_{g,n}(\varepsilon).\]
This orbit space is compact. Recall that the \emph{extended mapping class group} of a $n$-pointed surface $S$ of genus $g$, denoted $\widehat{\Mod}_{g,n}$, is the group of isotopy class of (not necessarily orientation preserving) diffeomorphisms of $S$ that leave the marked points invariant.  The topology of the truncated Teichm\"{u}ller space is described by the following result due to Ji and Wolpert.

\begin{prop}[Ji--Wolpert \cite{JiWolpert_ConfiniteSpace}]\label{prop:DefRetractTeich}
Let $\varepsilon>0$ be given.  If $\varepsilon<\delta$, then there exists a $\widehat{\Mod}_{g,n}$-equivarient deformation retraction from $\Teich_{g,n}$ onto $\Teich_{g,n}(\varepsilon)$.
\end{prop}

We may similarly define $\Teich_{g,2n}^\sigma(\varepsilon)$ and $\sM_{g,2n}^\sigma(\varepsilon)$.  Let $\sC^\sigma(S)$ denote the simplicial complex whose $k$-simplices are $(k+1)$-tuples of $\sigma$-orbits of isotopy classes of $\sigma$-invariant essential simple closed curves in $S$ that may be realized disjointly.  We have that
\[ \Teich_{g,2n}^\sigma(\varepsilon) = \{ \mathscr{X} \in \Teich_{g,2n}^\sigma \mid \ell_\mathscr{X}(a) \geq \varepsilon \hspace{4pt} \forall a \in \sC^\sigma(S)\}\]
and
\[ \sM_{g,2n}^\sigma(\varepsilon) = \Mod_{g,2n}^\sigma \backslash \Teich_{g,2n}^\sigma(\varepsilon).\]

\begin{rem}\label{rem:TeichSigmaContract}
Notice that $\Teich_{g,2n}^\sigma(\varepsilon)$ is also given by the fixed locus of $\Psi_\sigma$ in $\Teich_{g,2n}(\varepsilon)$.  From this observation and Proposition \ref{prop:DefRetractTeich}, we have that $\Teich_{g,2n}^\sigma(\varepsilon)$ is contractible for $\varepsilon$ sufficiently small.
\end{rem}

We claim that $\sM_{g,2n}^\sigma(\varepsilon)$ is compact for $\varepsilon$ sufficiently small.  The proof of this result is analogous to the proof of compactness of $\sM_{g,n}(\varepsilon)$ (see \cite[12.4]{FarbMargalit_Primer}). Before stating and proving this claim, we note the following result due to Sepp\"al\"a.

\begin{thm}[Sepp\"al\"a \cite{Seppala_BersConstant}]\label{thm:BersConstant}
Let $S$ be a surface of genus $g$ with $2n$ marked points and an orientation reversing involution $\sigma$.  Suppose that $2-2g-2n<0$.  There exists a constant $L$ that is dependent on $S$ such that for each hyperbolic surface $X'$ that is homeomorphic to $S'$ there exists a $\sigma$-invariant pants decomposition $\{c_i\}_i$ of $X'$ such that $\ell_{X'}(c_i)\leq L$ for each $i$.
\end{thm}

\begin{prop}\label{prop:MumfordCompactness}
Let $\varepsilon>0$ be given.  If $\varepsilon<\delta$ and $\varepsilon<L$, then the space $\sM_{g,2n}^\sigma(\varepsilon)$ is compact.
\end{prop}
\begin{proof}
Recall that $\Teich_{g,2n}^\sigma(\varepsilon)$ is a metric space and thus $\sM_{g,2n}^\sigma(\varepsilon)$ inherits the topology of a metric space from $\Teich_{g,2n}^\sigma(\varepsilon)$.  So to prove compactness, it suffices to prove sequential compactness.

Let $\varepsilon>0$ be given and let $\{X_i\}$ be a sequence in $\sM_{g,2n}^\sigma(\varepsilon)$.  Let $\mathscr{X}_i$ denote a lift of $X_i$ to $\Teich_{g,2n}^\sigma$.  By Theorem \ref{thm:BersConstant}, there exists a $\sigma$-invariant pants decomposition $\sP_i$ such that $\ell_{\mathscr{X}_i}(\gamma) \in [\varepsilon,L]$ for each $\gamma \in \sP_i$.  There exists a finite number of topological types of pants decompositions.  So there exists a subsequence, which we will also denote by $\mathscr{X}_i$, and homeomorphisms $f_i \in \Mod_{g,2n}$ such that $f_i(\sP_i) = \sP_1$.  Now we may fix Fenchel--Nielsen coordinates with respect to $\sP_1$.  Notice that the $f_i \cdot \mathscr{X}_i$ have length parameters in $[\varepsilon,L]$.  Each of the twist parameters can be chosen to lie in $[0,2\pi]$.  The subsequence $\mathscr{X}_i$ lie in a compact subset of Euclidean space and consequently there exists a convergent subsequence.  The desired result follows.
\end{proof}

\subsection{A Subspace at Infinity}\label{subsect:TeichTheory_Infinity}
These truncated moduli spaces provide a means for us to define a notion of infinity for moduli spaces of real algebraic curves.  Recall that given a hyperbolic structure $\mathscr{X}$ in $\Teich_{g,2n}^\sigma$ and fixed $M \in \IR_{>0}$, the set of isotopy classes of simple closed curves in $S$ with lengths $\ell_{\mathscr{X}}\leq M$ is finite \cite[12.3.1]{FarbMargalit_Primer}.  From this fact, we have the following proposition.

\begin{prop}\label{prop:Exhaustion}
Let $S$ be a surface of genus $g$ with $2n$ marked points and an orientation reversing involution $\sigma$.  For $\varepsilon$ sufficiently small, the subspaces $\sM_{g,2n}^\sigma(\varepsilon)$ exhaust $\sM_{g,2n}^\sigma$, that is,
\[ \sM_{g,2n}^\sigma = \bigcup_{\varepsilon>0} \sM_{g,2n}^\sigma(\varepsilon).\]
\end{prop}

For fixed $\varepsilon<\delta$, we set
\[ \sV_{g,2n}^\sigma := \Teich_{g,2n}^\sigma - \Teich_{g,2n}^\sigma(\varepsilon)\]
and
\[\sU_{g,2n}^\sigma := \sM_{g,2n}^\sigma - \sM_{g,2n}^\sigma(\varepsilon).\footnote{Perhaps it would be more precise to write $\sU_{g,2n}^\sigma(\varepsilon) := \sM_{g,2n}^\sigma - \sM_{g,2n}^\sigma(\varepsilon)$; however, since the homotopy type of $\sM_{g,2n}^\sigma - \sM_{g,2n}^\sigma(\varepsilon)$ is independent of $\varepsilon$ for $\varepsilon<\delta$, we choose to omit $\varepsilon$ from our notation and simply write $\sU_{g,2g}^\sigma$.  We justify the notation for $\sV_{g,2n}^\sigma$ in the same manner.} \]
In light of Proposition \ref{prop:Exhaustion}, the subspace $\sU_{g,2n}^\sigma$ of $\sM_{g,2n}^\sigma$ may be viewed as the subspace of $\sM_{g,2n}^\sigma$ at infinity.

\section{Complexes Associated to Surfaces}\label{sect:CurveComplexes}

It is well-known that the topology of the boundary of $\Teich_{g,n}(\varepsilon)$ is encoded by the curve complex of a surface $S$ of genus $g$ with $n$ marked points.  Recall that the \emph{curve complex} associated to an orientable surface $S$ is the simplicial complex whose $k$-simplices are $(k+1)$-tuples of isotopy classes of essential simple closed curves in $S$ that may be realized disjointly.  As we will show in Lemma \ref{lem:PropOfThick}, the topology of the boundary of $\Teich_{g,2n}^\sigma(\varepsilon)$ is encoded by the intersection patterns of $\sigma$-orbits of $\sigma$-invariant essential simple closed curves in a surface $S$ of genus $g$ with $2n$ marked points and an orientation reversing involution $\sigma$.  In this section, we discuss complexes associated to surfaces.  First, we discuss the curve complex for an arbitrary surface (possibly with boundary and possibly non-orientable) and its topology.  Second, we define a new complex associated to a surface called the $\abc$-complex and we compute its homotopy type.

For this section, let $F$ be a surface (possibly with boundary and possibly non-orientable).  A \emph{simple (closed) curve} $\alpha$ in $F$ is an embedding $\alpha: I \to F$ ($\alpha: S^1 \to F$).  We say that a simple curve in $F$ is \emph{essential} if it is not isotopic to a curve in a regular neighborhood of a point, puncture, or boundary component.  A curve that is not essential is said to be \emph{nonessential}.  An \emph{essential arc} in $F$ is an essential curve $\alpha$ in $F$ such that $\alpha(0)$ and $\alpha(1)$ lie in the boundary of $F$ or in the collection of marked points of $F$.  A simple closed curve is said to be \emph{two-sided} if a regular neighborhood of the curve is an annulus.  A simple closed curve is said to be \emph{one-sided} if a regular neighborhood of the curve is a M\"{o}bius band.  A curve $\mu$ in a surface $F$ is called a \emph{M\"{o}bius curve} if one component of $F - \mu$ is a M\"{o}bius band, that is, $\mu$ is the boundary of an embedded M\"{o}bius band in $F$.  Notice that there is a bijective correspondence between isotopy classes of one-sided curves and isotopy classes of M\"{o}bius curves.  Indeed, for each isotopy class of one-sided curves there is a unique isotopy class of M\"{o}bius curves given by the boundaries of the regular neighborhoods of the one-sided curves.

Finally, since the geometric realization functor does not change the homotopy type of a simplicial complex, we abusively switch back and forth between simplicial complexes and their geometric realizations throughout the remainder of this paper.

\subsection{The Curve Complex}\label{subsect:CurveComplexes_Curve}

We generalize our definition of the curve complex to include non-orientable surfaces.

\begin{defn}\label{defn:CurveComplex}
The \emph{curve complex} associated to a surface $F$, denoted $\sC(F)$, is the simplicial complex whose vertices are the isotopy classes of essential simple closed curves in $F$, excluding M\"{o}bius curves.  A set of $k+1$ vertices $\{v_0,\dots,v_k\}$ defines a $k$-simplex if the geometric intersection number $i(v_i,v_j)$ is zero for all $i$ and $j$ in $\{0,\dots,k\}$.
\end{defn}

An obvious question is why we do not include M\"{o}bius curves in the curve complex.  As we will see in Remark \ref{rem:MobiusLiftingArg}, M\"{o}bius curves are not needed to encode the topology of the boundary of $\Teich_{g,2n}^\sigma(\varepsilon)$.  Moreover, the homotopy type of $\sC(F)$ is independent of whether or not isotopy classes of M\"{o}bius curves are included in $\sC(F)$.

Combining the results of Harer and Ivanov, we have the following statement about the connectivity of the curve complex of a surface.

\begin{thm}[Harer \cite{Harer_VirtualCohomDim}, Ivanov \cite{Ivanov_VirtualCohomDim}]\label{thm:CCConnectivity}
Let $F$ be a surface with $g$ handles, $c$ cross caps, and $n$ marked points.  The curve complex of $F$ is $d(g,c,n)$-connected where

\[ d(g,c,n) = \begin{cases}
n-3 		& \mbox{if    }\,\,\,g = 0 = c, \\
2g+c-1	& \mbox{if    }\,\,\,n = 0, \\
2g+c+n-2	& \mbox{if    }\,\,\,n>0,g>0,c=0, \\
2g+c+n-2	& \mbox{if    }\,\,\,n>0,g=0, c>3.\\ \end{cases}\]
\end{thm}

To compute the homotopy type of the curve complex, we need the following lemma.

\begin{lem}\label{lem:CCHomDim}
Let $F$ be a surface with $g$ handles, $c$ cross caps, and $n$ marked points.  The curve complex of $F$ is not contractible and its homological dimension is less than or equal to $d(g,c,n)$ where
\[ d(g,c,n) = \begin{cases}
n-4 		& \mbox{if    }\,\,\,g = 0 = c, \\
2g+c-2	& \mbox{if    }\,\,\,n = 0, \\
2g+c+n-3	& \mbox{if    }\,\,\,n>0,g>0,c=0, \\
2g+c+n-3	& \mbox{if    }\,\,\,n>0,g=0, c>3.\\ \end{cases}\]
\end{lem}
\begin{proof}
The case where $F$ is orientable is proven by Harer in \cite{Harer_VirtualCohomDim}.  So it suffices to prove the result for non-orientable surfaces.  Our argument is a generalization of Harer's argument to possibly non-orientable surfaces.

Given a surface $F$ with $g$ handles, $c$ cross caps, and $n$ boundary components, define the invariant $N_F = 2c+3g+n-3$.  Suppose by way of induction that the result holds for surfaces $F$ with invariant $N_F \leq N$.  If $N_F = 0$, then either $c = 0$ or $c \neq 0$.  If $c = 0$, then by the work of Harer we are done.  If $c \neq 0$, then $c=1$ and $s=1$ or $s=3$.  It follows that $\sC(F)$ is a point or empty, respectively.  In either case, the desired result holds.

To handle the inductive case let $\sC(F)^\circ$ denote the barycentric subdivision of $\sC(F)$.  The vertices of $\sC(F)^\circ$ correspond to multicurves\footnote{A \emph{multicurve} in a surface $F$ is a finite collection of isotopy classes of simple closed curves in $F$ that may be realized disjointly.} and faces are added for chains of proper inclusions among these multicurves.  Given a vertex $v \in \sC(F)^\circ$, we define the \emph{weight of $v$}, denoted $w(v)$, to be the number of curves contained in the multicurve $v$ minus one.  Equivalently, $w(v)$ is the dimension of the simplex in $\sC(F)$ for which $v$ is a barycenter.

Let $X_k$ denote the full subcomplex of $\sC(F)^\circ$ composed of vertices with weights greater than or equal to $k$.  We will construct $\sC(F)^\circ$ by consecutively adding vertices of lower weight and show for each $k$ that $X_k$ is a wedge of spheres and has dimension $D$, where 
\[ D \leq \begin{cases} c+n-3 & n>0, \\ c-2 & n=0. \\ \end{cases}\]
Let $d$ denote the dimension of a top dimensional maximal simplex in $\sC(F)$.  Notice that $X_d$ is a discrete set of points and is thus a wedge of zero spheres and has dimension less than $D$.  Suppose by induction that $X_{k+1}$ has been shown to be a wedge of spheres all of dimensions less than $D$.  Suppose that $v = \{ a_0,\dots,a_k \} \in X_{k}-X_{k+1}$.  Assume that $v$ has a non-empty link.  Let $\{ F_i\}_{i=1}^t$ be the set of connected components of $F - \cup_j a_j$.  Notice that the link of $v$, denoted $L(v)$, in $X_{k+1}$ is simply the joins of $\sC(F_1),\dots, \sC(F_t)$.  Without loss of generality, suppose $F_1,\dots, F_\ell$ are orientable surfaces with $g_1,\dots,g_\ell$ handles and $r_1,\dots,r_\ell$ boundary components, respectively.  Suppose that $F_{\ell+1},\dots,F_t$ are non-orientable surfaces with $c_{\ell+1},\dots,c_t$ cross caps and $r_{\ell+1},\dots,r_t$ boundary components, respectively.  Via an Euler characteristic argument, we have
\begin{align*}
-D & = \sum_{i=1}^t \chi(F_i) 
\\ & = \left( \sum_{i=1}^\ell 2-2g_i-r_i \right) + \left( \sum_{i = \ell+1}^t 2-c_i -r_i \right)
\\ & = 2t - \left( \sum_{i=1}^\ell 2g_i+n_i \right) - \left( \sum_{i = \ell+1}^t c_i +r_i \right).
\end{align*}

By the work of Harer \cite{Harer_VirtualCohomDim}, $i\leq \ell$ implies that each $F_i$  is a wedge of spheres of dimensions less than $D$.  Inductively, we have that for $i > \ell$ each $F_i$ is a wedge of spheres of dimensions less than $D$.  Explicitly, we have that
\[ \dim(\sC(F_i) ) = 2g_i+r_i-3 <D \gap \mbox{for}\gap i \leq \ell \]
and
\[\dim(\sC(F_i)) \leq c_i+r_i-3 <D \gap \mbox{for} \gap i > \ell.\]  
Recalling that $L(v)$ is the joins of $\sC(F_1),\dots,\sC(F_t)$, we conclude that
\begin{align*} 
\dim(L(v)) & \leq \left( \sum_{i=1}^\ell 2g_i+r_i-3 \right) + \left( \sum_{i=\ell+1}^t c_i +r_i -3 \right) + (t-1) 
\\ & = \left( \sum_{i=1}^\ell 2g_i+r_i \right) + \left( \sum_{i=\ell+1}^t c_i + r_i \right) -3t+t-1 
\\ &= D-1.
\end{align*}
Notice that $X_k$ is obtained from $X_{k+1}$ by forming the join of $v$ with its link in $X_{k+1}$ for all $v$ in $X_{k}-X_{k+1}$.  

To complete the proof, we need to show that $X_k$ is not contractible.  Consider two vertices $v$ and $w$ in $X_k-X_{k+1}$.  Since we are working with the barycentric subdivision, we have that $v = \{ a_0,\dots, a_k \}$ and $w = \{ b_0,\dots,b_k \}$, where $a_i$ and $b_j$ are isotopy classes of essential simple closed curves in $F$.  

If the geometric intersection number $i(a_i,b_j) \neq 0$ for some $i$ and $j$, then there does not exist a vertex $x$ in $X_{k+1}$ such that $v,w \subset x$ as tuples of isotopy classes of essential simple closed curves.  It follows that $L(v) \cap L(w) = \varnothing$.  If $i(a_i,b_j) = 0$ for all $i$ and $j$, then let $x$ denote the vertex of minimal weight, $w(x) = m$, such that $v,w \subset x$ as tuples of isotopy classes of essential simple closed curves.  Notice that $L(v) \cap L(w)$ is the link of $x$ in $X_{m+1}$.  Inductively, $L(x)$ as a subspace of $X_{m+1}$ is a non-contractible wedge of spheres and thus forming the joins of $v$ and $w$ to their respective links suspends $L(x)$, producing a non-contractible space.

If such an $x$ does not exist then our space is a disjoint collection of cones and thus a wedge of zero spheres.  If such an $x$ exists, then it follows that $X_k$ is not contractible and is homotopy equivalent to a wedge of spheres each of dimension less than or equal to $D$.  This completes the proof.
\end{proof}

\begin{rem}\label{rem:HarerGap}
Harer's original proof of Lemma \ref{lem:CCHomDim} in the orientable case does not resolve the possibility that $\sC(F)$ may be contractible.  Our argument in the proof of Lemma \ref{lem:CCHomDim} provides a solution to this problem that is formulated completely in terms of the combinatorics of $\sC(F)$.
\end{rem}

Combining Theorem \ref{thm:CCConnectivity} and Lemma \ref{lem:CCHomDim}, we have the following result.

\begin{cor}\label{cor:CCHomotopyType}
Let $F$ be a surface with $g$ handles, $c$ cross caps, and $n$ marked points.  The curve complex of $F$ is homotopy equivalent to a wedge of spheres of dimension $d(g,c,n)$ where 
\[ d(g,c,n) = \begin{cases}
n-4 		& \mbox{if    }\,\,\,g = 0 = c, \\
2g+c-2	& \mbox{if    }\,\,\,n = 0, \\
2g+c+n-3	& \mbox{if    }\,\,\,n>0,g>0,c=0, \\
2g+c+n-3	& \mbox{if    }\,\,\,n>0,g=0, c>3.\\ \end{cases}\]
\end{cor}

\subsection{The $\abc$-Complex}\label{subsect:CurveComplexes_abc}
To study the topology of the boundary of $\Teich_{g,2n}^\sigma(\varepsilon)$, we need to understand the intersection properties of simple closed curves, simple arcs that start and end in boundary components, and boundary components.  This intersection information is combinatorially encoded by the following complex.

\begin{defn}\label{defn:ABCComplex}
The \emph{$\abc$-complex} associated to a surface $F$, denoted $\abc(F)$, is the simplicial complex whose vertices are:
\begin{itemize}
\item the isotopy classes of simple closed curves in $F$ that are non-trivial and not M\"{o}bius curves;
\item the isotopy classes of essential simple arcs in $F$ relative to the boundary of $F$ such that each has end points in the boundary of $F$.
\end{itemize}
A set of $k+1$ vertices $\{v_0,\dots,v_k\}$ defines a $k$-simplex if the geometric intersection number $i(v_i,v_j)$ is zero for all $i$ and $j$.
\end{defn}

\begin{rem}\label{rem:ABCNomen}
The name $\abc$-complex was chosen because the $\abc$-complex encodes the intersection patterns of arcs ($\sA$), boundary components ($\sB$), and closed curves ($\sC$).
\end{rem}

\begin{rem}\label{rem:ABCIncludes} 
Notice that the $\abc$-complex includes some isotopy classes of nonessential simple closed curves.  Namely, the isotopy classes of the boundary components of $F$.  We also consider a subset of essential simple arcs in $F$.  Specifically, an isotopy class of an essential simple arc relative to the boundary of $F$, say $a$ in $F$, is a vertex in $\abc(F)$ if and only if for each representative $\alpha$ of $a$ we have $\alpha(0)$ and $\alpha(1)$ are in $\partial F$.  Hence, we do not consider arcs that have end points lying in the collection of marked points of $F$.  Observe that if $\partial F$ is empty, then $\abc(F) = \sC(F)$.
\end{rem} 

\begin{rem}\label{rem:MotivateABC}
To motivate Definition \ref{defn:ABCComplex}, consider a real algebraic curve $(X,\sigma)$.  Consider the underlying surface with involution, say $S$ and $\sigma$, and the quotient map $S \to S/\sigma$.  Each simplex in $\sC^\sigma(S)$\footnote{Recall that $\sC^\sigma(S)$ is the simplicial complex whose set of $k$-simplices consists of $(k+1)$-tuples of $\sigma$-orbits of isotopy classes of $\sigma$-invariant essential simple closed curves in $S$ that may be realized disjointly.}, which is a $\sigma$-orbit of $\sigma$-invariant curves, descends to a simplex of $\abc(S/\sigma)$ under the action of $\sigma$.  Conversely, any simplex of $\abc(S/\sigma)$ may be lifted to a simplex in $\sC^\sigma(S)$.  These operations define a simplicial isomorphism between $\sC^\sigma(S)$ and $\abc(S/\sigma)$.  Hence, $\abc(S/\sigma)$ combinatorially encodes the intersection properties of isotopy classes of $\sigma$-invariant essential simple closed curves in $S$.  
\end{rem}

\begin{rem}\label{rem:MobiusLiftingArg}
Remark \ref{rem:MotivateABC} shows why we do not consider M\"{o}bius curves in our curve complexes.  Since the orientation double cover of a M\"{o}bius band is an annulus, a representative of the isotopy class of a M\"{o}bius curve $\mu$ in $S/\sigma$ lifts to two disjoint, isotopic simple closed curves in $S$.  Hence, $\mu$ does not lift to a simplex in $\sC^\sigma(S)$ and consequently does not encode the desired combinatorial information.
\end{rem}

There is a nice relationship between the topology of the $\abc$-complex and the topology of the underlying curve complex that is given by the following proposition.

\begin{prop}\label{prop:ABCSuspensionThm}
Let $F$ be a surface with $r \geq 1$ boundary components.  If $\sC(F)$ is non-empty, then $\abc(F)$ is homotopy equivalent to the join of an $({r-1})$-sphere with $\sC(F)$.
\end{prop}

The proof of Proposition \ref{prop:ABCSuspensionThm} uses a technique that was developed by Hatcher \cite{Hatcher_Triangulations} to show that the arc complex of a surface is contractible.

\begin{proof}
Observe that $\abc(F)$ has a natural filtration.  List the boundary components of $F$ as $\beta_1,\dots,\beta_r$.  Denote their respective isotopy classes by $b_1,\dots,b_r$.  Define $X_{k}$ to be the full subcomplex of $\abc(F)$ spanned by the vertices of $\sC(F)$, $\{b_1,\dots,b_k\}$, and all isotopy classes of essential arcs in $\abc(F)$ whose representatives have boundaries in $\cup_{i=1}^k \beta_i$.  It follows that 
\[\sC(F) = X_0 \subset X_1 \subset \cdots \subset X_{r-1} \subset X_r = \abc(F).\]
We claim that $X_k$ is homotopy equivalent to the suspension of $X_{k-1}$.  After proving this claim, the result follows.

Fix some vertex $v$ in $X_{k}- X_{k-1}$ with representative $\alpha$ such that $\alpha$ is an essential simple arc in $F$ with $\alpha(0)$ and $\alpha(1)$ lying in $\beta_{k}$.  Let $p$ be an arbitrary point in a simplex of $X_k$ spanned by vertices $v_0,\dots,v_m$ such that $v_i \neq \beta_k$ for all $i$.  Equivalently, $p$ is not in the simplicial star of $\beta_k$.  Express $p$ in barycentric coordinates, that is, write $p = \sum_i c_iv_i$ where $\sum_i c_i = 1$ and $c_i > 0$ for all $i$.  We will define a flow from $p$ into the simplicial star of $v$, denoted $St(v)$.  Given that our choice of $p$ is arbitrary, we may define similar flows for all points in $X_{k} - St(\beta_k)$.  Given that our flow will be the identity on $St(\beta_k)$ and that the closure of $St(\beta_k)$ is the cone on $X_{k-1}$, we see that the image of the deformation retraction is a suspension of $X_{k-1}$ with suspension points $v$ and $\beta_k$.

We realize each $v_i$ on $F$ by one of its representative curves $\alpha_i$.  We realize $p$ on $F$ by thickening each curve $\alpha_i$ to a width of $c_i$, forming an $\alpha_i$-band.  Via isotopies, we make the union of the $\alpha_i$-bands to be a closed interval that is disjoint from $\beta_k$.  In other words, we rearrange the $\alpha_i$-bands along $\alpha$ to look like the left-hand side of Figure \ref{fig:HatcherFlowF}. 

\begin{figure}[h]
   \centering
   \includegraphics[width=4in]{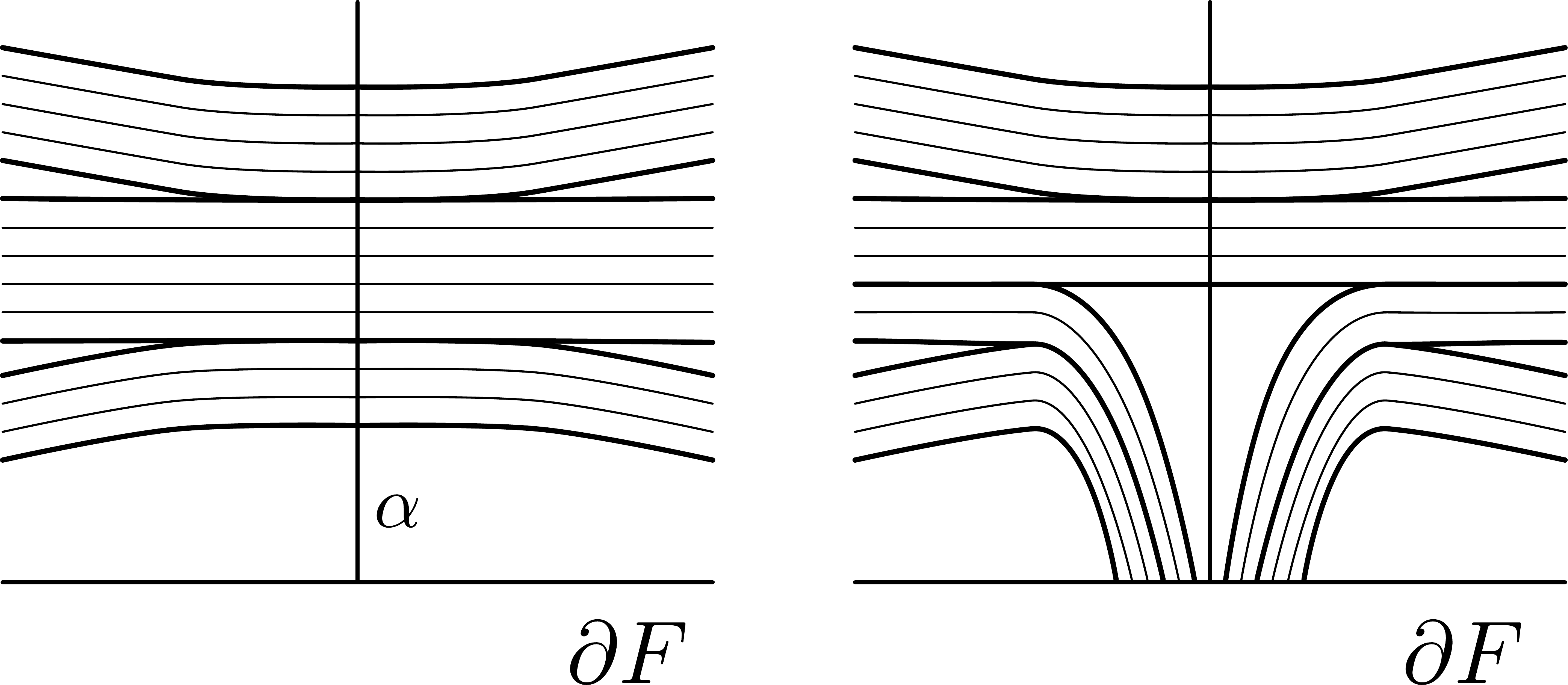} 
   \caption{Realized flow of $p$ into the simplicial star of $v$.}
   \label{fig:HatcherFlowF}
\end{figure}

The thickness of this union of $\alpha_i$-bands is $\omega = \sum_i c_i i(\alpha_i,\alpha)$.  We define a flow for the point $p$ by considering the value $\omega t$.  Where at $t=t_0$, we push the union of the $\alpha_i$-bands $\omega t_0$ along a fixed direction that follows the image of $\alpha$ on $F$.  At $t=0$, the realized point in the simplex of $X_k- St(\beta_k)$ is $p$.  At $0<t<1$, the arrangement of the $\alpha_i$-bands resembles the right-hand side of Figure \ref{fig:HatcherFlowF} and is now given by a new collection of bands.  Each band corresponds to an isotopy class of curves and the thickness of the band gives the desired barycentric coordinate.  Note, if $\alpha_i$ is essential, then at least one of the two new curves obtained from the $\alpha_i$-band is essential.  If one of these curves is nonessential, we send the coordinates from this curve to zero.  This allows for the flow to be well-defined.  At $t=1$, the realized point $p'$ is contained in a simplex composed of arcs disjoint from $\alpha$ and thus $p'$ is in the simplicial star of $v$.  This flow is well-defined on intersections of simplices and is continuous on simplices.  In fact, it is linear on simplices.  Thus, it defines a deformation retraction of $X_k- St(\beta_k)$ onto the simplicial star of $v$ and gives the desired suspension of $X_{k-1}$.
\end{proof}

Combining Corollary \ref{cor:CCHomotopyType} and Proposition \ref{prop:ABCSuspensionThm} gives the homotopy type of the $\abc$-complex of a surface.

\begin{thm}\label{thm:ABCHomotopyType}
Let $F$ be a surface with $g$ handles, $c$ cross caps, $r\geq 1$ boundary  components, and $n$ marked points.  The complex $\abc(F)$ is homotopy equivalent to a wedge of spheres of dimension $d(g,c,r,n)$ where
\[d(g,c,r,n) = \begin{cases}
2r+n-4		& \mbox{if    }\,\,\, g = 0 = c,\\
2g+c+2r+n-3	& \mbox{if    }\,\,\, g>0,c=0,\\
2g+c+2r+n-3	& \mbox{if    }\,\,\, g=0,c>3.\\
\end{cases}\]
\end{thm}

\section{Virtual Cohomological Dimension of Mapping Class Groups}\label{sect:VCD}

In this section, we study a group cohomological property of mapping class groups of surfaces with orientation reversing involutions.  We will show that the homology and cohomology of $\Mod_{g,2n}^\sigma$ satisfy duality relations analogous to those for compact manifolds.  Specifically, we will show that $\Mod_{g,2n}^\sigma$ is a virtual duality group.

\subsection{Virtual Duality Groups}\label{subsect:VCD_VDG}
The notion of a duality group that we consider was introduced by Bieri and Eckmann \cite{BieriEckmann_GroupHomDuality}.

\begin{defn}\label{defn:DualityGroup}
A group $G$ is called a \emph{duality group} of dimension $m$ if there is a (right) $G$-module $C$ (called a \emph{dualizing module}) and an element $e \in H_m(G;C)$ (called a \emph{fundamental class}) such that the cap-product with $e$ determines an isomorphism
\[ (e \frown -): H^k(G;A) \to H_{m-k}(G;C\otimes A)\]
for every (left) $G$-module $A$ and for all $k$.
\end{defn}

We will be particularly interested in virtual duality groups.

\begin{defn}\label{defn:VirtualDualityGroup}
A group $G$ is called a \emph{virtual duality group} of dimension $m$ if there exists a finite index subgroup $H$ of $G$ that is a duality group of dimension $m$.
\end{defn}

We will apply the following theorem to $\Mod_{g,2n}^\sigma$ to prove Theorem \ref{thm:InvolutionDualityGroup}.

\begin{thm}[Bieri, Eckmann \cite{BieriEckmann_GroupHomDuality}]\label{thm:GroupHomDuality}
Let $G$ be a group and $H \trianglelefteq G$ be a finite index subgroup admitting an Eilenberg-MacLane complex $X$, which is a compact, connected, orientable manifold with boundary of dimension $m$.  Let $\tilde{X}$ be the universal cover of $X$.  If the (reduced) integral homology groups $H_i(\partial \tilde{X})$ are equal to 0 for $i \neq q$ and if $H_q(\partial \tilde{X})$ is torsion-free, then $G$ is a virtual duality group of dimension $m-q-1$.
\end{thm}

\subsection{Proof of Theorem A}
Let $\Mod_{g,n}^\sigma$ be the group $G$ defined in Theorem \ref{thm:GroupHomDuality}.  It is well-known that there exists $\Gamma' \trianglelefteq \Mod_{g,2n}^\sigma$ that is a finite index subgroup of $\Mod_{g,n}^\sigma$ that acts freely on $\Teich_{g,2n}^\sigma(\varepsilon)$.  Let ${\Teich}_{g,2n}^\sigma(\varepsilon)$ be the space $X$ in Theorem \ref{thm:GroupHomDuality}.  We need the following lemma.

\begin{lem}\label{lem:PropOfThick}
Let $S$ be surface of genus $g$ with $2n$ marked points and an orientation reversing involution $\sigma$.  The space $\Teich_{g,2n}^\sigma(\varepsilon)$ has the following properties:
\begin{enumerate}
\item $\Gamma'$ acts properly discontinuously on $\Teich_{g,2n}^\sigma(\varepsilon)$.
\item $\Gamma'$ acts virtually freely on $\Teich_{g,2n}^\sigma(\varepsilon)$.
\item $\Gamma' \backslash \Teich_{g,2n}^\sigma(\varepsilon)$ is a compact, connected, orientable manifold of dimension $3g-3+2n$.
\item $\Teich_{g,2n}^\sigma(\varepsilon)$ is contractible.
\item $\partial \Teich_{g,2n}^\sigma(\varepsilon)$ is homotopy equivalent to $\abc(S/\sigma)$.
\end{enumerate}
\end{lem}

\begin{proof}
Properties (1) and (2) are obvious.  The statement of compactness in property (3) follows via the same argument used in the proof of Proposition \ref{prop:MumfordCompactness}.  The statement of orientability in property (3) follows by replacing $\Gamma'$ with an index two subgroup if necessary.  Property (4) was noted in Remark \ref{rem:TeichSigmaContract}.  It remains to prove property (5).

The proof of property (5) reduces to showing that $\abc(S/\sigma)$ is homotopy equivalent to a nerve of a covering $\{C_\tau\}_{\tau \in \sC^\sigma(S)}$ of $\partial \Teich_{g,2n}^\sigma(\varepsilon)$ by closed sets $C_\tau$.  To this end, set
\[C_\tau = \{\mathscr{X} \in\Teich_{g,2n}^\sigma(\varepsilon) \mid \ell_{\mathscr{X}}(a) = \varepsilon \,\,\,\, \mbox{for} \,\,\,\, a \in \tau \in \sC^\sigma(S)\}\]
where $\tau$ is a simplex in $\sC^\sigma(S)$ and $a$ is a vertex of $\tau$.  Using Fenchel--Neilsen coordinates, it is clear that
\[\partial \Teich_{g,2n}^\sigma(\varepsilon) = \bigcup_{\tau \in \sC^\sigma(S)} C_\tau.\]  
Notice that if $\cap_{a \in \tau} C_a \neq \varnothing$, then $\tau$ must give a simplex in $\sC^\sigma(S)$.  By Remark \ref{rem:MotivateABC}, we know that the homotopy type of $\sC^\sigma(S)$ is the same as the homotopy type of $\abc(S/\sigma)$.  To obtain the desired result, we simply need to show that $C_\tau \cap C_{\tau'}$ is contractible (if non-empty) for all $\tau$ and $\tau'$ in $\sC^\sigma(S)$.  Consider
\[ T = \{ \mathscr{X} \in \Teich_{g,2n}^\sigma \mid \ell_{\mathscr{X}}(\alpha) = \varepsilon = \ell_{\mathscr{X}}(\beta) \,\,\,\, \mbox{for all} \,\,\,\, \alpha \in \tau,\, \beta \in \tau' \}.\]
Let $\IR^k$ correspond to the twist parameters of all curves in $\tau$ and $\tau'$.  Since we have selected $\varepsilon <\delta$ (recall Proposition \ref{prop:ShortGeodesics}), we can use Fenchel--Nielsen coordinates to identify $T/\IR^k$ with the fixed locus of the involution $\Psi_\sigma$ in the Teichm\"{u}ller space of $F$, where $F$ is the complement of $S$ by the curves in $\tau$ and $\tau'$.  We must have that $(C_\tau \cap C_{\tau'})/\IR^k$ is given by the truncation of the fixed locus of the involution $\Psi_\sigma$ in the Teichm\"{u}ller space of $F$ for some $\varepsilon<\delta$.  Consequently, $(C_\tau \cap C_{\tau'})/\IR^k$ is contractible.  It follows that $C_\tau \cap C_{\tau'}$ is contradictable and the desired result follows.
\end{proof}

Theorem \ref{thm:InvolutionDualityGroup} immediately follows from Theorem \ref{thm:ABCHomotopyType}, Theorem \ref{thm:GroupHomDuality}, and Lemma \ref{lem:PropOfThick}.

\section{Orbifold Homotopy Groups of Moduli Spaces}\label{sect:OrbiHomotopyGroups}
In this section, we show that the $m$th orbifold homotopy groups of $\sM_{g,2n}^\sigma$ relative to $\sU_{g,2n}^\sigma$ vanish for all $m$ less than or equal to a computable integer which is a topological invariant of $(S,\sigma)$.

\subsection{Orbifold Homotopy Groups}
Let $X$ be a topological space and let $\Gamma$ be a group that acts properly discontinuously (but not necessarily freely) on $X$.  Using the Borel construction, let
\[p: E\Gamma \times X \to E \Gamma \times_{\Gamma} X \] 
denote the canonical covering space obtained from the (free) diagonal action of $\Gamma$ on $E \Gamma \times X$.  

\begin{defn}
The \emph{$m$th orbifold homotopy group of the space $X/\Gamma$}, denoted $\pi_n^{orb}(X/\Gamma)$, is defined as
\[ \pi_m^{orb}(X/\Gamma) := \pi_m(E\Gamma \times_\Gamma X).\]
\end{defn}

We will be interested in the case of $X = {\Teich}_{g,2n}^\sigma$ and $\Gamma = \Mod_{g,2n}^\sigma$.

\subsection{Proof of Theorem B}

To prove Theorem \ref{thm:OrbiHomotopyGroups}, we will need the following lemma.

\begin{lem}\label{lem:ConnInf}
Let $S$ be a surface of genus $g$ with $2n$ marked points and an orientation reversing involution $\sigma$.  Suppose that $2-2g-2n<0$ and $\sC^\sigma(S)$ is connected.  If $\mathscr{X}$ and $\mathscr{Y}$ are in $\sV_{g,2n}^\sigma$, then there exists a path in $\sV_{g,2n}$ from $\mathscr{X}$ to $\mathscr{Y}$.
\end{lem}
\begin{proof}
Fix some $\varepsilon<\delta$ (where $\delta$ is given in Proposition \ref{prop:ShortGeodesics}) and assume that $\sV_{g,2n}^\sigma$ and $\sU_{g,2n}^\sigma$ are defined using this $\varepsilon$.  Since $\mathscr{X}$ and $\mathscr{Y}$ are in $\Teich_{g,2n}^\sigma - \Teich_{g,2n}^\sigma(\varepsilon)$, there exists $\sigma$-orbits of isotopy classes of essential simple closed curves, say $a^\sigma$ and $b^\sigma$, such that $\ell_{\mathscr{X}}(a) < \varepsilon$ for all $a \in a^\sigma$ and $\ell_{\mathscr{Y}}(b)< \varepsilon$ for all $b \in b^\sigma$.  Since $\sC^\sigma(S)$ is connected, there exists a sequence of $\sigma$-orbits of isotopy classes of $\sigma$-invariant essential simple closed curves
\[ a^\sigma = c_1^\sigma, \dots, c_m^\sigma = b^\sigma\]
such that the geometric intersection number $i(c_i,c_{i+1}) = 0$ for each $c_{i} \in c_i^\sigma$ and each $c_{i+1} \in c_{i+1}^\sigma$ and for all $i$.

Choose Fenchel--Nielsen coordinates such that $c_1^\sigma$ and $c_2^\sigma$ are in the corresponding pants decomposition and their lengths give the first two coordinates when $\Teich_{g,2n}^\sigma$ is identified with $\IR^{3g-3+2n}$ under this choice of coordinates.  Suppose that $\mathscr{X}$ is given by the tuple $(x_1,x_2,x_3,\dots,x_{3g-3+2n})$ under this choice of coordinates.

Consider the path $\alpha_2: I \to \Teich_{g,2n}^\sigma - \Teich_{g,2n}^\sigma(\varepsilon) \subset \Teich_{g,2n}^\sigma \cong \IR^{3g-3+2n}$ given by
\[ \alpha_2(t) = \left(x_1, (1-t)x_2+\frac{\varepsilon}{42}(t),x_3,\dots,x_{3g-3+2n}\right).\]
By construction $\alpha_2(0) = \mathscr{X}$ and $\alpha_2(1) = \mathscr{X}_2$ such that $\ell_{\mathscr{X}_2}(c_2)<\varepsilon$ for each $c_2 \in c_2^\sigma$.  Repeating this procedure for $c_2^\sigma$ and $c_3^\sigma$ and so on, we obtain paths $\alpha_2,\dots,\alpha_m$.  Concatenating these paths, we obtain a path $\alpha = \alpha_2\cdots \alpha_m$\footnote{We denote path concatenation/multiplication from left to right.} in $\Teich_{g,2n}^\sigma - \Teich_{g,2n}^\sigma(\varepsilon)$ where $\alpha(0) = \mathscr{X}$ and $\alpha(1) = \mathscr{X}_m$ such that $\ell_{\mathscr{X}_m}(b)<\varepsilon$ for each $b \in b^\sigma$.  Selecting Fenchel--Nielsen coordinates such that $b^\sigma$ is in the corresponding pants decomposition, we may take the straight-line path, denoted $\beta$, from $\mathscr{X}_m$ to $\mathscr{Y}$ under this choice of coordinates.  The concatenation $\alpha\beta$ produces the desired path described in the statement of the proposition.
\end{proof}

\begin{rem}\label{rem:ConnDiffState}
One may interpret Lemma \ref{lem:ConnInf} as follows: If the images of $\mathscr{X}$ and $\mathscr{Y}$ in $\sM_{g,2n}^\sigma$ lie in $\sU_{g,2n}^\sigma$, then there exists a path in $\Teich_{g,2n}^\sigma$ from $\mathscr{X}$ to $\mathscr{Y}$ whose image in $\sM_{g,2n}^\sigma$ lies in $\sU_{g,2n}^\sigma$.
\end{rem}

\begin{rem}
The assumption that $\sC^\sigma(S)$ is connected is essential to prove Lemma \ref{lem:ConnInf}.  Barring a finite number of cases, $\sC^\sigma(S)$ is connected.  More precisely, if $\Fix(\sigma)$ is empty, then $\sC^\sigma(S)$ is connected if and only if $g \geq 3$.  This follows from the isomorphism between $\sC^\sigma(S)$ and $\sC(S/\sigma)$ and the conditions for $\sC(S/\sigma)$ to be connected.  If $\Fix(\sigma)$ is non-empty, then it is easy to see that $\sC^\sigma(S)$ is connected.
\end{rem}

We introduce the following notation for the remainder of this paper:
\begin{align*}
& T := E \Mod_{g,2n}^\sigma \times {\Teich}_{g,2n}^\sigma,
\\ & M := E \Mod_{g,2n}^\sigma \times_{\Mod_{g,2n}^\sigma} {\Teich}_{g,2n}^\sigma,
\\ & V := E \Mod_{g,2n}^\sigma \times \sV_{g,2n}^\sigma,
\\ & U := E \Mod_{g,2n}^\sigma \times_{\Mod_{g,2n}^\sigma} \sV_{g,2n}^\sigma.
\end{align*}

Now we give the proof of Theorem \ref{thm:OrbiHomotopyGroups}.

\begin{proof}
Since $T \to M$ and $V \to U$ are coverings, we have isomorphisms
\[\pi_m(T) \cong \pi_m(M) \gap \mbox{and} \gap \pi_m(V) \cong \pi_m(U)\]
for each $m \geq 2$.  By Theorem \ref{thm:ABCHomotopyType} and Lemma \ref{lem:PropOfThick}, we have that $V$ is homotopy equivalent to a wedge of spheres of dimension $d$, where $d$ denotes the homological dimension of $\abc(S/\sigma)$.  Since $T$ is contractible, it follows that $\pi_m(M)$ and $\pi_m(U)$ are trivial for $2\leq m < d$.

It remains to show that the inclusion $U \hookrightarrow M$ induces an isomorphism on fundamental groups.  Let $\alpha$ be a loop in $M$ with base point in $U$.  We may lift $\alpha$ to a path $\tilde{\alpha}$ in $T$ with end points lying in $V$.  By Lemma \ref{lem:ConnInf}, there exists a path $\tilde{\beta}$ in $V$ such that $\tilde{\beta}(0) = \tilde{\alpha}(0)$ and $\tilde{\beta}(1) = \tilde{\alpha}(1)$.  Since $T$ is simply-connected, there exists a homotopy from $\tilde{\alpha}$ to $\tilde{\beta}$ that leaves the end points fixed.  This homotopy descends to a homotopy from $\alpha$ to the projection of $\tilde{\beta}$, which by construction lies entirely in $U$ (see Remark \ref{rem:ConnDiffState}).  By definition, we have that the inclusion $U \hookrightarrow M$ induces an isomorphism on fundamental groups.  This completes the proof of the theorem.
\end{proof}

\bibliographystyle{plain}
\bibliography{VCDRACBib.bib}

\begin{thebibliography}{10}

\bibitem{Abikoff_RealAnalyticTeichTheory}
William Abikoff.
\newblock {\em The real analytic theory of {T}eichm\"uller space}, volume 820
  of {\em Lecture Notes in Mathematics}.
\newblock Springer, Berlin, 1980.

\bibitem{BieriEckmann_GroupHomDuality}
Robert Bieri and Beno Eckmann.
\newblock Groups with homological duality generalizing {P}oincar\'e duality.
\newblock {\em Invent. Math.}, 20:103--124, 1973.

\bibitem{BuserSeppala_SymmetricPantsDecompostionsOfRiemannSurfaces}
Peter Buser and Mika Sepp\"al\"a.
\newblock Symmetric pants decompositions of {R}iemann surfaces.
\newblock {\em Duke Math. J.}, 67(1):39--55, 1992.

\bibitem{FarbMargalit_Primer}
Benson Farb and Dan Margalit.
\newblock {\em A primer on mapping class groups}, volume~49 of {\em Princeton
  Mathematical Series}.
\newblock Princeton University Press, Princeton, NJ, 2012.

\bibitem{GouldenHarerJackson_VirtualEulerChar}
I.~P. Goulden, J.~L. Harer, and D.~M. Jackson.
\newblock A geometric parametrization for the virtual {E}uler characteristics
  of the moduli spaces of real and complex algebraic curves.
\newblock {\em Trans. Amer. Math. Soc.}, 353(11):4405--4427, 2001.

\bibitem{Harer_VirtualCohomDim}
John~L. Harer.
\newblock The virtual cohomological dimension of the mapping class group of an
  orientable surface.
\newblock {\em Invent. Math.}, 84(1):157--176, 1986.

\bibitem{Harnack_HarnackInequality}
Axel Harnack.
\newblock Ueber die {D}arstellung einer willk\"urlichen {F}unction durch die
  {F}ourier-{B}essel'schen {F}unctionen.
\newblock {\em Math. Ann.}, 35(1-2):41--62, 1889.

\bibitem{Hatcher_Triangulations}
Allen Hatcher.
\newblock On triangulations of surfaces.
\newblock {\em Topology Appl.}, 40(2):189--194, 1991.

\bibitem{Ivanov_VirtualCohomDim}
N.V. Ivanov.
\newblock Complexes of curves and the teichm\"{u}ller modular group.
\newblock {\em Russian Mathematics Survey}, 42(3):55--107, 1987.

\bibitem{JiWolpert_ConfiniteSpace}
Lizhen Ji and Scott~A. Wolpert.
\newblock A cofinite universal space for proper actions for mapping class
  groups.
\newblock In {\em In the tradition of {A}hlfors-{B}ers. {V}}, volume 510 of
  {\em Contemp. Math.}, pages 151--163. Amer. Math. Soc., Providence, RI, 2010.

\bibitem{Kobayashi_TransformationGroups}
Shoshichi Kobayashi.
\newblock {\em Transformation groups in differential geometry}.
\newblock Classics in Mathematics. Springer-Verlag, Berlin, 1995.
\newblock Reprint of the 1972 edition.

\bibitem{Seppala_BersConstant}
M.~Sepp\"al\"a.
\newblock Moduli spaces of stable real algebraic curves.
\newblock {\em Ann. Sci. \'Ecole Norm. Sup. (4)}, 24(5):519--544, 1991.

\bibitem{SeppalaSilhol_ModuliSpacesForRealAlgebraicCurvesAndRealAbelianVarieties}
M.~Sepp\"al\"a and R.~Silhol.
\newblock Moduli spaces for real algebraic curves and real abelian varieties.
\newblock {\em Math. Z.}, 201(2):151--165, 1989.

\bibitem{Seppala_ComplexAlgebraicCurvesWithRealModuli}
Mika Sepp\"al\"a.
\newblock Complex algebraic curves with real moduli.
\newblock {\em J. Reine Angew. Math.}, 387:209--220, 1988.

\bibitem{Weichold_InvolutionTopType}
G.~Weichold.
\newblock \"{U}ber symmetrische riemannsche fl\"{a}chen und die
  periodizit\"{a}tsmodulen der zugerh\"{o}rigen abelschen normalintegrale
  erstes gattung.
\newblock Leipziger dissertation, 1883.

\end{thebibliography}

\end{document}